\documentclass[a4paper,12pt]{article}
\usepackage{amsmath,amssymb,enumerate,amscd,amsthm,hyperref}
\usepackage[top=2.7cm,bottom=2.7cm,left=2.8cm,right=2.8cm]{geometry}
\newtheorem{theorem}{Theorem}
    \newtheorem{corollary}[theorem]{Corollary}
    \newtheorem{lemma}[theorem]{Lemma}
    
    \newtheoremstyle{def} 
    {\topsep}                    
    {\topsep}                    
    {}                   
    {}                           
    {\bfseries}                   
    {.}                          
    {.5em}                       
    {}  
\theoremstyle{def}
\newtheorem{remark}[theorem]{Remark}
\newtheorem{definition}[theorem]{Definition}
\newcommand{\oper}[1]{\operatorname{#1}}
\newcommand{\ir}[1]{\mathcal{#1}}

\newcommand{\RR}{\mathbb{R}}

\newcommand{\inv}{^{-1}}
\newcommand{\iso}{\operatorname{iso}}
\newcommand{\parcv}[2]{\frac{\partial#1}{\partial#2}}

\author{B.~Aradi, D.~Cs.~Kert\'esz}       
\title{A characterization of holonomy invariant functions on tangent bundles} 
\begin{document}
\maketitle
%
%
\begin{abstract}
    We show that the holonomy invariance of a function on the tangent bundle of a manifold, together with very mild regularity conditions on the function, is equivalent to the existence of local parallelisms compatible with the function in a natural way. Thus, in particular, we obtain a characterization of generalized Berwald manifolds. We also construct a simple example of a generalized Berwald manifold which is not Berwald.
\end{abstract}

AMS \emph{Subject Class.}\ (2010):
    53B05,  
    53B40. 

\section{Introduction}
A function given on the tangent bundle of a manifold is said to be holonomy invariant if there is a covariant derivative on the manifold whose parallel translations preserve the function. The Finsler function of a generalized Berwald manifold is an example of such a function. So is, in particular, the Finsler function of a Berwald manifold, in which case the covariant derivative is torsion-free and unique.

Berwald manifolds have been studied intensely; many equivalent definitions and characterizations are known (see, e.g., \cite{SZLK11}), and there is a nice classification of this type of Finsler manifolds due to the structure theorem of Szab\'o \cite{Sza81}. Such a classification of generalized Berwald manifolds is not yet known, nevertheless many interesting papers have been written on the subject, for example, by Hashiguchi and Ichijy\=o \cite{HaIc82}, Ichijy\=o \cite{Ich76.1,Ich76.2}, Szak\'al and Szilasi \cite{SzaSzi01}, Tam\'assy \cite{TAM00} and Vincze \cite{Vin,Vin2}.

The present work was strongly motivated by the papers \cite{Ich76.1,Ich76.2} of Ichijy\=o, in which he proved that the connected generalized Berwald manifolds are the same as the so-called $\{V,H\}$-manifolds. Ichijy\=o was interested in `Finsler manifolds modeled on a Minkowski space', that is, Finsler manifolds such that the tangent spaces are `isometrically linearly isomorphic' to a single Minkowski space. He introduced the slightly stronger concept of a $\{V,H\}$-manifold, consisting of a vector space $V$ endowed with a Minkowski norm (or a Finsler norm, as we prefer to call it) and a manifold with an $H$-structure (in the sense of a $G$-structure), where $H$ is a Lie subgroup of $\oper{GL}(V)$ leaving the Minkowski norm invariant. Such a manifold can be endowed with a Finsler function which is modeled on the Minkowski space $V$. One can use the $H$-compatible local trivializations of the tangent bundle to transfer the Minkowski norm of $V$ to the tangent spaces. The so obtained Finsler function is well-defined, because the transition mappings between $H$-compatible trivializations preserve the Minkowski norm by assumption. The surprising result of Ichijy\=o was that $\{V,H\}$-manifolds are no more general than generalized Berwald manifolds.

It is worth noting that the Finsler function constructed on a $\{V,H\}$-manifold is locally a one-form Finsler function. Indeed, each $H$-compatible local trivialization can be identified with a local co-frame $(\alpha_i)_{i=1}^n$, then our Finsler function is locally of the form $F=f\circ(\alpha_1,\dots,\alpha_n)$, where $f$ is a Minkowski norm on $\RR^n$. For a systematic study of one-form Finsler functions, see \cite{MaShi}.

Hashiguchi suggested (Problem 9 in \cite{Hash83}) that one should define $\{V,H\}$-manifolds under weaker conditions, more precisely, that the conditions on the Finsler function are too strong. In this paper we generalize Ichijy\=o's concept. We consider an arbitrary function on the tangent manifold compatible with a covering parallelism (Definition~\ref{parallcomp}). We use parallelisms instead of an $H$-structure for conceptual simplicity only, so if the function is in particular a Finsler function, our notion is equivalent to that of $\{V,H\}$-manifolds.

Using our new definition, we reformulate and also generalize Ichijy\=o's theorem: instead of the strong regularity conditions imposed on Finsler functions, we require only continuity and a kind of definiteness. Under such mild assumptions we prove that the function is holonomy invariant if, and only if, it is compatible with a covering parallelism on the manifold (Theorem~\ref{chargenber}). As a corollary, by applying this result to a Finsler function, we obtain a characterization of generalized Berwald manifolds (Corollary~\ref{genber}), analogous to Ichijy\=o's result.

The structure of the paper is as follows. In Section~\ref{sec:prel} we introduce our notation and conventions, and we also recall some basic facts concerning parallelisms. The next section is devoted to the preparations required for the proof of our main result in Section~\ref{sec:res}. Finally, we present a simple example of a non-Berwaldian generalized Berwald manifold.
%
%
\section{Preliminaries}\label{sec:prel}

Throughout the paper, by a \emph{manifold} we mean a smooth manifold of dimension $n$ \mbox{($n\geq 2$)}, whose underlying topological space is Hausdorff, second countable and connected. The \emph{tangent bundle} of a manifold $M$ is $\tau\colon TM\to M$.

By a \emph{curve} in a manifold we shall always mean a regular smooth curve whose domain is an open interval containing $0$.

Consider a curve $\gamma\colon I \to M$. A \emph{vector field along} $\gamma$ is a smooth mapping $X$ from $I$ to $TM$ such that $\tau\circ X=\gamma$.
A covariant derivative $\nabla$ on $M$ induces a covariant derivative $\nabla_\gamma$ on the $ C^\infty(I)$-module of vector fields along $\gamma$ such that for every $t\in I$ we have $\nabla_\gamma X(t):=\nabla_{\dot\gamma(t)}\bar X$, where $\dot\gamma(t)$ is the velocity of $\gamma$ at $t$, and $\bar X$ is a vector field on $M$ such that (locally) $\bar X\circ\gamma=X$.

A vector field $X$ along $\gamma$ is \emph{parallel} (\emph{with respect to} $\nabla$) if it satisfies the ordinary differential equation $\nabla_\gamma X=0$. The \emph{parallel translation} along $\gamma$ from $\gamma(0)$ to $\gamma(t)$ is the mapping
\[
P_\gamma^{t}\colon T_{\gamma(0)}M\to T_{\gamma(t)}M,\ v\mapsto X(t),
\]
where $X$ is the unique parallel vector field along $\gamma$ such that $X(0)=v$. As is well-known, this mapping is a linear isomorphism between the tangent spaces. Later we simply write $P_\gamma$ for $P_\gamma^1$ if $I$ contains $1$.
\bigskip

Let $\pi\colon \ir P\to M\times M$ be the vector bundle over $M\times M$ whose fibre at a point $(p,q)$ is the real vector space $\oper{Hom}(T_pM,T_qM)$. A \emph{parallelism} on $M$ is a smooth section $P$ of this vector bundle satisfying
\[
P(r,q)\circ P(p,r)=P(p,q) \ \mbox{ and }\ P(p,p)=1_{T_pM}
\]
for all $p,q,r\in M$ (see\ \cite[p.~174]{GHV72}). These conditions imply that the mappings
\[
P(p,q)\colon T_pM\to T_qM, \ (p,q)\in M\times M
\]
are actually bijective.

Most manifolds do not admit a parallelism. Exactly those manifolds share this property, which can be equipped with a global frame field.
These manifolds are said to be \emph{parallelizable}. However, any point in a manifold has an open neighbourhood, which is, as an open submanifold, parallelizable.
Sometimes for a parallelism $P$ on an open submanifold $\ir U$ of $M$ we use the notation $(\ir U,P)$. A vector field $X$ on $\ir U$ is called \emph{$P$-parallel} if $X(q)=P(p,q)(X(p))$ for any two points $p,q$ in $\ir U$.

By a \emph{covering parallelism} of a manifold $M$ we mean a family $(\ir U_\alpha,P_\alpha)_{\alpha\in\ir A}$ of parallelisms, where $(\ir U_\alpha)_{\alpha\in\ir A}$ is an open covering of $M$.

A parallelism $(\ir U,P)$ induces a trivialization $\varphi$ of $TM$ over $\ir U$, given by
\[
\varphi\colon(q,v)\in \ir U\times \RR^n\mapsto \varphi(q,v):=P(p,q)\circ\eta_p(v)\in TM,
\]
where $p$ is a fixed point in $\ir U$ and $\eta_p$ is an arbitrary linear isomorphism from $\RR^n$ to $T_pM$. (Note that $\varphi$ depends on $p$ and $\eta_p$.) Then for any two points $q$ and $r$ in $\ir U$ we have
\begin{equation}\label{paraltriv}
  P(q,r)\circ\varphi_q=\varphi_r,
\end{equation}
where $\varphi_q$ stands for the mapping $v\in\RR^n\mapsto \varphi_q(v):=\varphi(q,v)\in T_qM$.
%
%
\section{Compatibility notions and auxiliary results}

If $M$ is a manifold and $F\colon TM \to \RR$ is any function, we use the notation $F_p$ for the restriction $F\upharpoonright T_pM$ ($p\in M$).

In this section we introduce a natural notion of compatibility of such functions with a covariant derivative and a parallelism. Roughly speaking, `compatibility' means here that the linear isomorphisms (between the tangent spaces) induced by the given additional structure on $M$ leave the function $F$ invariant. For example, given a Finsler manifold $(M,F)$ and a covariant derivative $\nabla$ on $M$ we can ask whether the induced parallel translations preserve the Finsler norms of tangent vectors.

Now the precise definitions:

\begin{definition}\label{nablacomp}
Let $\nabla$ be a covariant derivative on a manifold $M$ and $F$ a function on $TM$. We say that $F$ is \emph{holonomy invariant with respect to} $\nabla$, or $F$ is \emph{compatible with} $\nabla$, if the parallel translations induced by $\nabla$ preserve $F$, that is, for any curve $\gamma\colon I\to M$ and parameter $t\in I$ we have
  \[
  F_{\gamma(t)}\circ P^t_\gamma=F_{\gamma(0)}.
  \]
\end{definition}

\begin{definition}\label{parallcomp}
A function $F$ on $TM$ is \emph{compatible with a parallelism} $P$ on $M$ if $F$ takes the same value on parallel vectors,
that is, for any $p,q\in M$ the relation
  \[
  F_q\circ P(p,q)=F_p
  \]
holds. The function $F$ is \emph{compatible with a covering parallelism} $(\ir U_\alpha,P_\alpha)_{\alpha\in\ir A}$ if the restriction
of $F$ to $\tau\inv(\ir U_\alpha)$ is compatible with the parallelism $(\ir U_\alpha,P_\alpha)$ for all $\alpha\in\ir A$.
\end{definition}

In Section \ref{sec:res} we will show that for a very general class of functions on $TM$ the compatibility with a covariant derivative and with a covering parallelism are equivalent properties. In the remainder of this section we develop some technical results required for the proof.

Our first observation is that the compatibility of a function on $TM$ and a parallelism $P$ can be expressed also in terms of a trivialization induced by $P$:
\begin{lemma}\label{lmm:Pf}
 If a function $F\colon TM\to \RR$ is compatible with a parallelism $(\ir U,P)$ and $\varphi$ is a local trivialization of $TM$ over $\ir U$ induced by $P$,
then there exists a function $f$ on $\RR^n$ such that $f=F_p\circ \varphi_p$ for all $p\in\ir U$.
\end{lemma}

\begin{proof}
Consider the diagram
  \[
  \begin{CD}
  \RR^n @>\varphi_p>> T_pM @> F_p>>\RR\\
  @V 1_{\RR^n}VV @VV P(p,q)V @VV 1_\RR V\\
  \RR^n @>\varphi_q>> T_qM @> F_q>>\RR
  \end{CD}
  \]
  for some $p,q\in\ir U$. The left part of the diagram commutes by (\ref{paraltriv}), while the right part commutes by the compatibility of $F$ and $P$. Hence the entire diagram is commutative and we have $F_p\circ \varphi_p=F_q\circ \varphi_q$. Thus the function $F_p\circ\varphi_p$ is independent of the chosen point $p$ of $\ir U$, so we can set $f:=F_p\circ\varphi_p$.
\end{proof}

The next lemma is a mild generalization of a result of Ichijy\=o \cite{Ich76.1}.

\begin{lemma}\label{isoflie}
  Let $V$ be a finite dimensional real vector space, and let $f\colon V \to \RR$ be a continuous function which vanishes at $0$, and only there.
Then the `isometry group'
  \[
  \iso(f):=\{A\in\operatorname{End}(V)\mid f\circ A=f\}
  \]
  of $f$ is a Lie subgroup of $\oper{GL}(V)$.
\end{lemma}

\begin{proof}
Notice first that the elements of $\iso(f)$ are invertible. Indeed, for any $A$ in $\iso(f)$ and any vector $v$ in $V\setminus\{0\}$ we have $f\circ A(v)=f(v)\neq0$, thus $A(v)=0$ is impossible by our condition on $f$. So it follows that $\iso(f)$ is a subset of $\oper{GL}(V)$ and also that $\iso(f)$ is a subgroup of $\oper{GL}(V)$.

It remains to show that the subgroup $\iso(f)$ is closed, then Cartan's closed subgroup theorem implies that $\iso(f)$ is indeed a Lie group. To do this, consider a sequence $(A_k)$ in $\iso(f)$ and assume that it converges to $A\in\operatorname{End}(V)$. Then, taking into account the continuity of $f$, we obtain
  \[
  f(A(v))=f\left(\lim_{k\to\infty}A_k(v)\right)=\lim_{k\to\infty}f(A_k(v))=\lim_{k\to\infty}f(v)=f(v)
  \]
for any $v\in V$. This proves that $A\in \iso(f)$, whence $\iso(f)$ is closed in $\oper{GL}(V)$.
\end{proof}

Our third lemma can be found in \cite{Wendl08} as an exercise; for the reader's convenience we present it with a proof.

\begin{lemma}\label{lmm:isom}
Let $G$ be a Lie subgroup of $\oper{GL}(\RR^n)$, $\mathfrak g$ its Lie algebra, and let $A\colon I\to\mathfrak g$ be a curve.
If $\Phi\colon I\to\oper{GL}(\RR^n)$ is a solution of the initial value problem
\begin{equation}\label{eq:DE}
\Phi'(t)=A(t)\cdot\Phi(t),\quad \Phi(0)=\oper I_n,
\end{equation}
then it takes values only in $G$. (Here the dot stands for matrix multiplication, and $\oper I_n$ is the $n$ by $n$ identity matrix.)
\end{lemma}
\begin{proof}
  We show that (\ref{eq:DE}) implies that the curve $t\in I\mapsto(t,\Phi(t))\in\RR\times\oper{GL}(\RR^n)$ is an integral curve of a vector
field on $\RR\times G$, thus $\Phi$ must run in $G$.

  Since $\oper{GL}(\RR^n)$ is an open subset of $\oper M_n(\RR)$, we may identify its tangent manifold with $\oper{GL}(\RR^n)\times \oper M_n(\RR)$.
  If $\varrho_g$ denotes the right translation by $g$ in $\oper{GL}(\RR^n)$ and $R_{A(t)}$ is the right invariant vector field on $\oper{GL}(\RR^n)$ with $R_{A(t)}(\oper I_n)=(\oper I_n,A(t))$, then  we obtain
  \begin{align*}
   \dot\Phi(t)&=(\Phi(t),\Phi'(t))\overset{(\ref{eq:DE})}=(\Phi(t),A(t)\cdot\Phi(t))=(\varrho_{\Phi(t)}\oper I_n,\varrho_{\Phi(t)}A(t))\\
   &=\varrho_{\Phi(t)*}(\oper I_n,A(t))=R_{A(t)}(\Phi(t)).
  \end{align*}
  Thus $t\mapsto (t,\Phi(t))$ is an integral curve of the vector field
  \begin{equation}\label{eq:vfA}
  (t,g)\mapsto \left(1_t,R_{A(t)}(g)\right)
  \end{equation}
  on $\RR\times \oper{GL}(\RR^n)$. However, $R_{A(t)}$ is tangent to the submanifold $G$ of $\oper{GL}(\RR^n)$, and, obviously, (\ref{eq:vfA})
is tangent to $\RR\times G$, so the restriction of (\ref{eq:vfA}) to $\RR\times G$ is a vector field.
\end{proof}

\begin{remark}\label{rmrk:isom}
 The converse of the lemma is immediate: if $\Phi$ is a curve in $G$, then $\Phi'(t)=A(t)\cdot\Phi(t)$ for some curve $A$ in $\mathfrak g$.
\end{remark}
%
%
\section{The main result}\label{sec:res}

\begin{theorem}\label{chargenber}
  Let $F\colon TM\to \RR$ be a continuous function which is definite in the sense that $F(v)=0$ if, and only if, $v=0$. Then $F$ is holonomy invariant with respect to some covariant derivative on the manifold $M$ if, and only if, it is compatible with a covering parallelism.
\end{theorem}
Before the proof, we need a lemma which establishes a relation between compatible parallelisms and covariant derivatives.

Let $\nabla$ be a covariant derivative on an open subset $\ir U$ of $M$ and $(\ir U,P)$ a parallelism. For each $p\in\ir U$ and $v\in T_pM$, we define an endomorphism $(\nabla P)_v$
on $T_pM$ by $(\nabla P)_v(w):=\nabla_vX$, where $X$ is the unique $P$-parallel vector field with $X(p)=w$.

The Christoffel symbols of $\nabla$ with respect to a $P$-parallel frame field $(E_i)_{i=1}^n$ are the smooth functions $\Gamma{}^i_{jk}$ on $\ir U$ given by $\nabla_{E_j}E_k=\Gamma{}^i_{jk} E_i$. Then
\begin{equation}\label{eq:NablaP}
  (\nabla P)_v(w)=w^kv^j\Gamma{}^i_{jk}(p) E_i(p),\quad \mbox{where } v=v^jE_j(p),\ w=w^kE_k(p),
\end{equation}
(summation convention in force).

\begin{lemma}\label{lmm:compalg}
  Let $P$ be a parallelism on a manifold $\ir U$, and let $F\colon T\ir U\to\RR$ be a definite continuous function compatible with $P$. Then a covariant derivative $\nabla$ is compatible with $F$ if, and only if, the endomorphism $(\nabla P)_v$ is in the Lie algebra $\mathfrak i(F_{\tau(v)})$ of $\iso(F_{\tau(v)})$ for any  $v\in T\ir U$.
 \end{lemma}

\begin{proof}
We note first that $\iso(F_{\tau(v)})$ is a Lie group by Lemma~\ref{isoflie}, thus we can speak of its Lie algebra $\mathfrak i(F_{\tau(v)})$. Furthermore, since $\iso(F_{\tau(v)})$ is a closed submanifold of the vector space $\oper{End}(T_{\tau(v)}\ir U)$, the Lie algebra $\mathfrak i(F_{\tau(v)})$ can be regarded as a linear subspace of $\oper{End}(T_{\tau(v)}\ir U)$, so the statement $(\nabla P)_v\in\mathfrak i(F_{\tau(v)})$ also makes sense.

Let $\gamma\colon I\to\ir U$ be a curve, $\varphi$ a trivialization of $T\ir U$ induced by $P$ (see the end of Section~\ref{sec:prel}), and define the function $f:=F_{\gamma(0)}\circ\varphi_{\gamma(0)}$ on $\RR^n$. Our first aim is to show that $F$ is invariant under $P_\gamma^t$ for any parameter $t$ (cf.\ Definition~\ref{nablacomp}) if, and only if, the curve $\Phi\colon I\to \oper{GL}(\RR^n)$ given by
\begin{equation}\label{eq:Phi}
 \Phi(t):=\varphi_{\gamma(t)}\inv\circ P_\gamma^t\circ\varphi_{\gamma(0)}
\end{equation}
runs in $\iso(f)$. Indeed, since we also have $f=F_{\gamma(t)}\circ\varphi_{\gamma(t)}$ by Lemma~\ref{lmm:Pf}, equation (\ref{eq:Phi}) implies $f\circ\Phi(t)=F_{\gamma(t)}\circ P_\gamma^t\circ \varphi_{\gamma(0)}$ for each $t\in I$. If we compare this to the definition of $f$, we see that the relations $f\circ \Phi(t)=f$ and $F_{\gamma(t)}\circ P_\gamma^t=F_{\gamma(0)}$ are equivalent.

Next we show that $\Phi$ takes values only in $\iso(f)$ if, and only if, $(\nabla P)_{\dot\gamma(t)}$ is in $\mathfrak i(F_{\gamma(t)})$ for any $t\in I$. This will conclude the proof, since any vector in $T\ir U$ is the velocity of a curve in $\ir U$.

Consider a vector $w\in T_{\gamma(0)}\ir U$. We have $P_{\gamma}^t(w)=X(t)$, where $X$ is the unique vector field along $\gamma$ such that $\nabla_\gamma X=0$ and $X(0)=w$. Let $(E_i)_{i=1}^n$ be the $P$-parallel frame field on $\ir U$ given by $E_i(p):=\varphi(p,e_i)$. Then we can write $X=X^i(E_i\circ\gamma)$ and $\dot\gamma=(\dot\gamma)^i(E_i\circ\gamma)$ for some smooth functions $X^i$, $(\dot\gamma)^i$ on $I$, and for all $t\in I$ we have
\begin{align*}
  0&=\nabla_\gamma X(t)=\nabla_\gamma(X^i(E_i\circ\gamma))(t)\\
  &=(X^i)'(t)(E_i\circ\gamma)(t)+X^i(t)(\nabla P)_{\dot\gamma(t)}E_i(\gamma(t))\\
  &\hspace{-4pt}\overset{(\ref{eq:NablaP})}=\big((X^i)'(t)+(\dot\gamma)^j(t)X^k(t)\Gamma^i_{jk}(\gamma(t))\big)(E_i\circ \gamma)(t).
  \end{align*}
Let $\Phi(t)=(\Phi^i_j(t))$. By (\ref{eq:Phi}) and by $P_{\gamma}^t(w)=X(t)$ we obtain $w^l\Phi^i_l(t)=X^i(t)$, which, together with the calculation above, lead to
\[
 0=w^l(\Phi^i_l)'+w^l\Phi^k_l(\dot\gamma)^j(\Gamma^i_{jk}\circ\gamma),\quad i\in\{1,\dots,n\}.
\]
Since the vector $w$ is arbitrary, we see that $\Phi$ satisfies an ODE of the form (\ref{eq:DE}) with $A(t)=\big(-(\dot\gamma)^j(t)\Gamma^i_{jk}(\gamma(t))\big)$. Lemma~\ref{lmm:isom} and Remark~\ref{rmrk:isom} imply that $\Phi$ runs in $\iso(f)$ if, and only if, the matrices $\big(-(\dot\gamma)^j(t)\Gamma^i_{jk}(\gamma(t))\big)$ are in the Lie algebra $\mathfrak i(f)$ of $\iso(f)$ for each $t\in I$.

It remains to show that $((\dot\gamma)^j(t)\Gamma^i_{jk}({\gamma(t)}))\in\mathfrak i(f)$ and $(\nabla P)_{\dot\gamma(t)}\in\mathfrak i(F_{\gamma(t)})$ are equivalent for all $t\in I$. We consider $\mathfrak i(f)$ and $\mathfrak i(F_{\gamma(t)})$ as linear subspaces of $\oper M_n(\RR)$ and $\oper{End}(T_{\gamma(t)}\ir U)$, respectively. We have the linear isomorphism
\[
 \mathbf c\colon B\in \oper M_n(\RR)\mapsto \varphi_{\gamma(t)}\circ B\circ\varphi_{\gamma(t)}\inv\in\oper{End}(T_{\gamma(t)}\ir U).
\]
In fact, $\mathbf c$ is just the mapping $(B^i_k)\mapsto B^i_k E^k({\gamma(t)})\otimes E_i({\gamma(t)})$ (where $(E^i)_{i=1}^n$ is the dual frame of $(E_i)_{i=1}^n$), therefore
\begin{equation}\label{eq:ciso}
 {\bf c}((\dot\gamma)^j(t)\Gamma^i_{jk}({\gamma(t)}))=(\dot\gamma)^j(t)\Gamma^i_{jk}({\gamma(t)})E^k({\gamma(t)})\otimes E_i({\gamma(t)})\overset{(\ref{eq:NablaP})}=(\nabla P)_{\dot\gamma(t)}.
\end{equation}
One can easily check that ${\bf c}\upharpoonright\iso(f)$ is a group isomorphism from $\iso(f)$ to $\iso(F_{\gamma(t)})$, because  $F_{\gamma(t)}\circ\varphi_{\gamma(t)}=f$. Thus its derivative at the unit element is a linear isomorphism from $\mathfrak i(f)$ onto $\mathfrak i(F_{\gamma(t)})$. However, ${\bf c}$ is linear, so its derivative is itself. We conclude that ${\bf c}$ is a bijection from $\mathfrak i(f)$ onto $\mathfrak i(F_{\gamma(t)})$, hence (\ref{eq:ciso}) implies our claim.
\end{proof}

\begin{trivlist}
\item \emph{Proof of Theorem~\ref{chargenber}.}
  Consider a definite, continuous function $F\colon TM\to \RR$. Recall that our base manifold $M$ is connected.

  (1) First, let us assume that the function $F$ is compatible with a covariant derivative $\nabla$ on $M$. Fix a point $p\in M$ and a chart $(\ir U,u)$ around $p$ such that $u(\ir U)$ is convex in $\RR^n$. Now we construct a parallelism on $\ir U$. For an arbitrary point $q \in \ir U$ consider the parametrized line segment $c_q$ connecting $u(p)$ and $u(q)$. Then $\gamma_q:=u\inv \circ c_q$ is a curve in $\ir U$ connecting $p$ with $q$. Now let
\[
 P(p,q):=P_{\gamma_q},
\]
where $P_{\gamma_q}$ is the parallel translation along $\gamma_q$ with respect to $\nabla$. For any $q_1,q_2\in\ir U$ define $P(q_1,q_2)$ as
\[
 P(q_1,q_2):=P(p,q_2)\circ P(p,q_1)\inv.
\]
It can be checked easily that $P$ is a parallelism over $\ir U$; the smoothness follows from the smooth dependence on parameters of ODE solutions. It is also clear by the holonomy invariance of $F$ that for any $q,r\in \ir U$ we have
\[
  F_r\circ P(q,r)=F_r\circ P_{\gamma_r}\circ P\inv_{\gamma_q}=F_p\circ P_{\gamma_q}\inv=F_q,
\]
which means that $F$ is indeed compatible with $P$.

To obtain a covering parallelism of $M$, we can apply the same method for sufficiently many $p\in M$.

(2) In this part we assume that $F$ is compatible with a covering parallelism $(\ir U_\alpha, P_\alpha)_{\alpha\in \ir A}$ of $M$, and we construct a covariant derivative $\nabla$ compatible with $F$.

We define a covariant derivative $\nabla^\alpha$ on each $\ir U_\alpha$ by setting all of its Christoffel symbols zero (with respect to a $P_\alpha$-parallel frame field). Then for each $v\in \tau\inv(\ir U_\alpha)$ the endomorphism $(\nabla^\alpha P_\alpha)_v$ is zero.
These covariant derivatives are compatible with (the proper restrictions of) $F$ by Lemma~\ref{lmm:compalg}.

If $\ir U_\alpha$ and $\ir U_\beta$ intersect, and $v\in\tau\inv(\ir U_\alpha\cap\ir U_\beta)$, then the endomorphisms $(\nabla^\alpha P_\beta)_v$ and $(\nabla^\beta P_\alpha)_v$ are no longer zero in general, but  they are still in the Lie algebra $\mathfrak i(F_{\tau(v)})$ of $\iso(F_{\tau(v)})$,
since $F$ is holonomy invariant with respect to $\nabla^\alpha$ and $\nabla^\beta$ (over $\ir U_\alpha$ and $\ir U_\beta$, respectively).
Thus, if we choose a partition of unity $(f_\alpha)_{\alpha\in\ir A}$ subordinate to the covering $(\ir U_\alpha)_{\alpha\in\ir A}$, the covariant derivative  $\nabla:=f_\alpha\nabla^\alpha$ on $M$ still has the property that the endomorphisms $(\nabla P_\alpha)_v$ are in $\mathfrak i(F_{\tau(v)})$.
Hence, by Lemma~\ref{lmm:compalg} again, $\nabla$ is compatible with $F$ over each $\ir U_\alpha$.
However, if $F$ is invariant under the parallel translation along pieces of a curve, it is invariant along the entire curve, thus $F$ is holonomy invariant with respect to $\nabla$, and the proof is complete.\hfill$\square$
\end{trivlist}

As a special case of Theorem~\ref{chargenber}, we obtain a characterization of generalized Berwald manifolds. For our purposes the following definition of such  manifolds is the most convenient (cf., \cite{SzaSzi01}, Definition~4.1\ and Proposition~4.3).

\begin{definition}
  A Finsler manifold $(M,F)$ is said to be a \emph{generalized Berwald manifold} if there exists a covariant derivative $\nabla$ on the base manifold $M$, such that the parallel translations induced by $\nabla$ preserve the Finsler function $F$.
\end{definition}

This is just Definition \ref{nablacomp} choosing $F$ to be, in particular, a Finsler function, thus a Finsler manifold $(M,F)$ is a generalized Berwald manifold if $F$ is holonomy invariant with respect to some covariant derivative on $M$. Using our main result we can express this condition in terms of parallelisms.

\begin{corollary}\label{genber}
  A Finsler manifold is a generalized Berwald manifold if, and only if, the Finsler function is compatible with a covering parallelism.
\end{corollary}

\begin{remark}
All our results remain true in a more general setting. Let $\pi\colon E\to M$ be an arbitrary (real) vector bundle, and let $F\colon E\to \RR$ be a continuous function which is definite in the above sense. Also in this case it is possible to define the compatibility of $F$ with a covariant derivative on the vector bundle and with a covering parallelism (the latter can be defined on the analogy of the tangent bundle case), and it turns out again that these compatibility concepts are equivalent.
\end{remark}
%
%
\section{An example of a proper generalized Berwald\\ manifold}

In this section we present a simple example of a generalized Berwald manifold, which is not of Berwald type. The idea is to define a Finsler function on a manifold which is compatible with a unique covariant derivative, and to show that this particular covariant derivative has non-vanishing torsion.

Our example will be a two-dimensional Randers manifold. We are going to define the covariant derivative with the help of a global parallelism,
and heavily use that there is a natural correspondence between the set of global parallelisms and 2-frames on the manifold.
\medskip

\emph{(1) Construction of the Randers manifold and a compatible parallelism.}\quad Let us consider the two-dimensional manifold $\RR^2$ and its standard global chart $(\RR^2,(x,y))$.
Define a 2-frame on $\RR^2$ by
\[
  E_1:=x\parcv{}{x}+\parcv{}{y},\quad E_2:=-\parcv{}{x},
\]
and let $E^1:=dy$, $E^2:=-dx+x\,dy$ be its dual frame. Consider the Finsler norm $f:=\sqrt{4x^2+12y^2}-x$ on $\RR^2$. Then
\[
F:=f\circ(E^1,E^2)=\sqrt{4(dy)^2+12(-dx+x\,dy)^2}-dy
\]
is a Finsler function for $\RR^2$ of Randers type.

The frame field $(E_1,E_2)$ induces a parallelism
\[
P(p,q)(v):=E^1(v)E_1(q)+E^2(v)E_2(q)
\]
($p,q\in\RR^2$, $v\in T_p\RR^2$), which is compatible with the Finsler function $F$. Indeed, if $w:=P(p,q)(v)$, then $E^1(w)=E^1(v)$ and $E^2(w)=E^2(v)$, hence $F(w)=F(v)$.
\medskip

\emph{(2) Construction of a compatible covariant derivative.}\quad Let $\nabla$ be the covariant derivative on $\RR^2$ characterized by $\nabla E_1=\nabla E_2=0$. Then for any $p\in\RR^2$, $v\in T_p\RR^2$ the mapping
\[
  X_v\colon q\in \RR^2 \mapsto X_v(q):=P(p,q)(v):=E^1(v)E_1(q)+E^2(v)E_2(q)\in T_q\RR^2
\]
is a vector field on the plane satisfying $\nabla X_v=0$. Hence the parallel translation along a curve $\gamma\colon I\to \RR^2$ acts by
\[
P_\gamma^t(v)=X_v(\gamma(t))=P(\gamma(0),\gamma(t))(v) \ \mbox{ for } v\in T_{\gamma(0)}\RR^2.
\]
Since $F$ is compatible with the parallelism $P$, it follows that $F$ is holonomy invariant with respect to $\nabla$.
Therefore $(\RR^2,F)$ is a generalized Berwald manifold.
\medskip

\emph{(3) There is no other covariant derivative compatible with $F$.}\quad Notice first that the isometry group of $F_p$ has only two elements for any $p\in\RR^2$. More precisely, in the basis $(E_1(p),E_2(p))$, the elements of $\oper{iso}(F_p)$ are represented by the matrices
\[
\left(\begin{array}{cc}1 & 0 \\0 & 1 \\\end{array}\right)\quad \mbox{ and }\quad \left(\begin{array}{cc}1 & 0 \\0 & -1 \\\end{array}\right).
\]
Indeed, if we assume that a linear mapping $A\colon \RR^2\to \RR^2$ is an isometry of the Finsler norm $f:=\sqrt{4x^2+12y^2}-x$,
then the four conditions that $A$ preserves the norms of the vectors $(1,0)$, $(-1,0)$, $(0,1)$ and $(0,-1)$ imply that $A$ is either the identity or the reflection about the axis $y=0$.

Now suppose that $F$ is holonomy invariant with respect to another covariant derivative $\bar\nabla$, and let $\gamma\colon I\to\RR^2$ be a curve. Then for the parallel translation $\bar P_\gamma^t$ we have $(\bar P_\gamma^t)\inv\circ P_\gamma^t\in\iso(F_{\gamma(0)})$. The parallel translations are smooth, hence the linear automorphism $(\bar P_\gamma^t)\inv\circ P_\gamma^t$ depends continuously on $t$. Since $(\bar P_\gamma^0)\inv\circ P_\gamma^0=1_{T_{\gamma(0)}\RR^2}$, it follows that $P_\gamma^t=\bar P_\gamma^t$ for all $t\in I$. Then $\nabla=\bar\nabla$, because a covariant derivative is determined by its induced parallel translations (see, e.g., \cite[Proposition 6.1.59]{IFM}).
\medskip

\emph{(4) The covariant derivative $\nabla$ has non-vanishing torsion.}\quad Indeed,
\[
\mathsf{T}^\nabla(E_1,E_2)=\nabla_{E_1}E_2-\nabla_{E_2}E_1-[E_1,E_2]=-\parcv{}{x}.
\]
Thus $(\RR^2,F)$ is not a Berwald manifold.
%
%
\bigskip

{\bf Acknowledgements.}
    We wish to express our gratitude to our supervisor, J\'ozsef Szilasi for encouraging us to write this paper and for his guidance throughout the work. We are thankful to the referee and to Rezs\H{o} Lovas for their valuable suggestions.

%
%

\bigskip

\noindent\emph{Bernadett Aradi}\\                              
    MTA-DE Research Group ``Equations, Functions and Curves'',\\
    Hungarian Academy of Sciences\\
    and Institute of Mathematics, University of Debrecen\\ 
    H--4010 Debrecen, P.O. Box 12, Hungary\\        
    E-mail: bernadett.aradi@science.unideb.hu         
\medskip

\noindent\emph{D\'avid Csaba Kert\'esz}\\
    Institute of Mathematics, University of Debrecen\\ 
    H--4010 Debrecen, P.O. Box 12, Hungary\\        
    E-mail: kerteszd@science.unideb.hu

%
%
\end{document}